\documentclass[tikz,border=10pt]{article}
\UseRawInputEncoding
\usepackage[english]{babel}
\usepackage[utf8x]{inputenc}
\usepackage[T1]{fontenc}

\usepackage{tikz}
\usepackage{pgflibraryarrows}
\usepackage{pgflibrarysnakes}
\usetikzlibrary{decorations.markings}

\usepackage[a4paper,top=3cm,bottom=2cm,left=3cm,right=3cm,marginparwidth=1.75cm]{geometry}

\usepackage{amsmath}
\usepackage{graphicx}
\usepackage[colorinlistoftodos]{todonotes}
\usepackage[colorlinks=true, allcolors=blue]{hyperref} \usepackage{mathrsfs}
\usepackage{amssymb}
\usepackage{lineno}
\usepackage{bbding}
\usepackage{fancyhdr}\usepackage{graphics}
\usepackage{titlesec}\usepackage{titletoc}
\usepackage{indentfirst}
\usepackage{amsmath}\usepackage{amsfonts}\usepackage{latexsym}
\usepackage{eucal}\usepackage{epsfig}\usepackage{mathrsfs,amsmath,amsthm}
\usepackage{geometry}

\newtheorem{theorem}{Theorem}[section]
 \newtheorem{lemma}[theorem]{Lemma}
 \newtheorem{proposition}[theorem]{Proposition}
 \newtheorem{corollary}[theorem]{Corollary}
 \newtheorem{definition}[theorem]{Definition}

 \newcount\refno
\def\CC{{\mathbb C}}

 \def\RR{{\mathbb R}}
 \def\NN{{\mathbb N}}
 \def\SS{{\mathbb S}}

\def\SD{{\mathscr D}}
 \def\SF{{\mathscr F}}
 \def\SH{{\mathscr H}}
 \title{\bf Hilbert transform on the Dunkl-Hardy Spaces
 \thanks{1} \footnote{E-mail:
huzhuoran010@163.com[ZhuoRan Hu].}}

\author{{ZhuoRan Hu}\\
{\small Department of Mathematics, Capital Normal University}\\
{\small Beijing 100048, China}}

\begin{document}

\maketitle \setcounter{page}{1} \pagestyle{myheadings}
 \markboth{Hu}{Hilbert transform on the Dunkl-Hardy Spaces }

\begin{abstract}
For $p>p_0=\frac{2\lambda}{2\lambda+1}$ with $\lambda>0$, the Hardy space $H_{\lambda}^p(\RR_+^2)$  associated with the Dunkl transform $\mathcal{F}_\lambda$ and the Dunkl operator $D$ on the real line $\RR$, where $D_xf(x)=f'(x)+\frac{\lambda}{x}[f(x)-f(-x)]$, is the set of functions $F=u+iv$ on the upper half plane $\RR^2_+=\left\{(x, y): x\in\RR, y>0\right\}$, satisfying $\lambda$-Cauchy-Riemann equations: $ D_xu-\partial_y v=0$, $\partial_y u +D_xv=0$, and $\sup\limits_{y>0}\int_{\RR}|F(x+iy)|^p|x|^{2\lambda}dx<\infty$ in \cite{ZhongKai Li 3}. Then it is proved in \cite{hu} that the real Dunkl-Hardy Spaces $H_{\lambda}^p(\RR)$ for $\frac{1}{1+\gamma_\lambda}<p\leq1$ are Homogeneous Hardy Spaces. In this paper, we  will continue to   investigate $\lambda$-Hilbert transform on the real Dunkl-Hardy Spaces $H_{\lambda}^p(\RR)$  for $\frac{1}{1+\gamma_\lambda}<p\leq1$ with $\displaystyle{\gamma_\lambda=1/(4\lambda+2)}$ and extend the results of $\lambda$-Hilbert transform in\,\cite{ZhongKai Li 3}.

\vskip .2in
 \noindent
 {\bf 2000 MS Classification:} 42B20, 42B25, 42A38.

 \end{abstract}

 %%%%%%%%%%%%%%%%%%%%%%%%%%%%%%%%%%%%%%%%%%%%%%%%%%%%%%%%%%%%%%%%%%%%%%%%%%%%%%%%%%%%%%%%%%
\setcounter{page}{1}
%%%%%%%%%%%%%%%%%%%%%%%%%%%%%%%%%%%%%%%%%%%%%%%%%%%%%%%%%%%%%%%%%%%%%%%%%%%%%%%%%%%%%%%%%%

\section{ Introduction and preliminaries }

For $0<p<\infty$, $L_{\lambda}^p(\RR)$ is the set of measurable functions satisfying
$ \|f\|_{L_{\lambda}^p}=\Big(c_{\lambda}\int_{\RR}|f(x)|^p|x|^{2\lambda}dx\Big)^{1/p}$ $<\infty$,
$c_{\lambda}^{-1}=2^{\lambda+1/2}\Gamma(\lambda+1/2)$,
and $p=\infty$ is the usual $L^\infty(\RR)$ space.
For $\lambda\geq0$, The Dunkl operator on the line is:
$$D_xf(x)=f'(x)+\frac{\lambda}{x}[f(x)-f(-x)]$$
involving a reflection part. The associated Fourier transform for the Dunkl setting for $f\in L_{\lambda}^1(\RR)$ is given by:
\begin{eqnarray}\label{fourier}
(\SF_{\lambda}f)(\xi)=c_{\lambda}\int_{\RR}f(x)E_\lambda(-ix\xi)|x|^{2\lambda}dx,\quad
\xi\in\RR , \,\,f\in L_{\lambda}^1(\RR).
\end{eqnarray}
$E_{\lambda}(-ix\xi)$ is the Dunkl kernel
$$E_{\lambda}(iz)=j_{\lambda-1/2}(z)+\frac{iz}{2\lambda+1}j_{\lambda+1/2}(z),\ \  z\in\CC$$
where $j_{\alpha}(z)$ is the normalized Bessel function
$$j_{\alpha}(z)=2^{\alpha}\Gamma(\alpha+1)\frac{J_{\alpha}(z)}{z^{\alpha}}=\Gamma(\alpha+1)\sum_{n=0}^{\infty}\frac{(-1)^n(z/2)^{2n}}{n!\Gamma(n+\alpha+1)} .$$
Since $j_{\lambda-1/2}(z)=\cos z$, $j_{\lambda+1/2}(z)=z^{-1}\sin z$, it follows that $E_0(iz)=e^{iz}$, and $\SF_{0}$ agrees with the usual Fourier transform. We assume $\lambda>0$ in
what follows.
And the associated $\lambda$-translation in Dunkl setting is
\begin{eqnarray}\label{tau}
 \tau_y f(x)=c_\lambda
     \int_{{\RR}}(\SF_{\lambda}f)(\xi)E(ix\xi)E(iy\xi)|\xi|^{2\lambda}d\xi,
     \ \ x,y\in{\RR} .
\end{eqnarray}
The $\lambda$-convolution$(f\ast_{\lambda}g)(x)$ of two appropriate functions $f$ and $g$ on $\RR$ associated to the $\lambda$-translation $\tau_t$ is defined by
$$(f\ast_{\lambda}g)(x)=c_{\lambda}\int_{\RR}f(t)\tau_{x}g(-t)|t|^{2\lambda}dt.$$
The "Laplace Equation" associated with the Dunkl setting is given by:
$$(\triangle_{\lambda}u)(x, y)=\left(D_x^2+ \partial_y^2\right) u(x, y)=\left(\partial_x^2+ \partial_y^2\right)u+ \frac{\lambda}{x}\partial_xu-\frac{\lambda}{x^2}\left(u(x, y)-u(-x, y)\right).$$
A $C^2$ function $u(x, y)$ satisfying $\triangle_{\lambda}u=0$ is  $\lambda$-harmonic.
When u and v are $\lambda$-harmonic functions satisfying $\lambda$-Cauchy-Riemann equations:
\begin{eqnarray}\label{a c r0}
\left\{\begin{array}{ll}
                                    D_xu-\partial_y v=0,&  \\
                                    \partial_y u +D_xv=0&
                                 \end{array}\right.
\end{eqnarray}
the function F(z)=F(x,y)=u(x,y)+iv(x,y)\,(z=x+iy)\, is a $\lambda$-analytic function.
 We define the Complex-Hardy spaces $H^p_\lambda(\RR^2_+)$ to be the set of
$\lambda$-analytic functions F=u+iv on $\RR^2_+$ satisfying
$$\|F\|_{H^p_\lambda(\RR^2_+)}=\sup\limits_{y>0}\left\{c_{\lambda}\int_{\RR}|F(x+iy)|^p|x|^{2\lambda}dx \right\}^{1/p}<\infty.$$

The present paper is a successive work of  \,\cite{ZhongKai Li 3} and\, \cite{hu}. In \,\cite{ZhongKai Li 3} the Boundedness of $\lambda$-Hilbert transform is discussed  on $H^1_\lambda(\RR)$, and we will extended the result to $p<1$ in this paper.
 In section\,\ref{facts} of this paper, some necessary facts about
the harmonic analysis related to the Dunkl transform on $\RR$ are summarized, including the $\lambda$-translation, $\lambda$-Poisson integral,
Conjugate $\lambda$-Poisson integral, $\lambda$-Hilbert transform and some basic facts associated with the $\lambda$-analytic Dunkl-Hardy Spaces
$H^p_\lambda(\RR^2_+)$ and the real Dunkl-Hardy Spaces $H^p_\lambda(\RR)$. In section\,\ref{Hilbert} and section\,\ref{bounded Hilbert} of this paper, we will discuss the $\lambda$-Hilbert transform and we will obtain our $\mathbf{main\  result}$ of this paper in $\mathbf{Theorem}$\,\ref{t1} and $\mathbf{Theorem}$\,\ref{us9}, that we will show that the $\lambda$-Hilbert transform can be extended as a  bounded operator from $H^p_\lambda(\RR)$ to
$L^p_\lambda(\RR)$ with $(2\lambda)/(2\lambda+1)<p<\infty$ and from $H^p_\lambda(\RR)$ to $H^p_\lambda(\RR)$ with $1/(4\lambda+2)<p\leq 1$.

As usual, $\mathcal{B}_\lambda(\RR)$ denotes the set of Borel measures $d\mu$ on $\RR$ for which $\|d\mu\|_{\mathcal{B}_\lambda}=c_\lambda\int_\RR |x|^{2\lambda}|d\mu(x)|$ is finite.  $\SD(\RR)$ designates  the space of $C^{\infty}$ functions on $\RR$ with compact support, $C_c(\RR)$  the spaces of
 continuous function on $\RR$ with compact support, and
$\SS(\RR)$ the space of  $C^{\infty}$ functions on $\RR$ rapidly decreasing together with their derivatives.
$L_{\lambda,{\rm loc}}(\RR)$ is the set of locally integrable functions on $\RR$ associated with the measure $|x|^{2\lambda}dx$.
Throughout the paper,  we use $A\lesssim B$ to denote the estimate $|A|\leq CB$ for some  absolute universal constant $C>0$, which may vary from line to line,
$A\gtrsim B$ to denote the estimate $|A|\geq CB$ for some absolute universal constant $C>0$, $A\sim B$  to denote the estimate $|A|\leq C_1B$, $|A|\geq C_2B$ for some absolute universal constant $C_1, C_2$.

\section{Basic Facts  in  Dunkl setting on the real line $\RR$}\label{facts}

\subsection{The Dunkl Operator and the Dunkl Transform}
For sake of simplicity,  we set $|E|_{\lambda}=c_\lambda\int_E |x|^{2\lambda}dx$ for a measurable set $E\subset\RR$, and $\langle f, g\rangle_\lambda=c_\lambda\int_\RR f(x)g(x) |x|^{2\lambda}dx$ whenever the integral exists. The Dunkl operator on the real line is given by
$$Df(x)=f'(x)+\frac{\lambda}{x}[f(x)-f(-x)].$$
A direct computation shows that
$$
(D^2f)(x)=f''(x)+\frac{2\lambda}{x}f'(x)-\frac{\lambda}{x^2}
[f(x)-f(-x)].
$$
For $f\in C^1(\RR)$, an inverse operator of $D$ is given by
\begin{eqnarray}\label{D-operator-inverse-1}
f(x)=f(0)+\frac{x}{2}\left(\int_{-1}^1({\rm sgn}\,s)(Df)(sx)ds+
\int_{-1}^1(Df)(sx)|s|^{2\lambda}ds\right).
\end{eqnarray}
For $f\in \SS(\RR)$, the inverse operator of $D$ can be given by
\begin{eqnarray}\label{D-operator-inverse-2}
f(x)=\frac{1}{2}\int_{-\infty}^{\infty}[\tau_x(Df)](-t)({\rm
sgn}\,t)dt,
\end{eqnarray}
and from the Formula\,$Df(x)=f'(x)+\lambda\int_{|s|\leq1}f'(xs)ds=f'(x)-\lambda\int_{|s|\geq1}f'(xs)ds$, we could know that the operator $D$ is invariant in $\SD(\RR)$ and $\SS(\RR)$.
If $f,\phi\in \SS(\RR)$, or $f\in C^1(\RR)$ with $\phi\in \SD(\RR)$,
or $f\in C_c^1(\RR)$  with $\phi\in C^{\infty}(\RR)$, we could get
\begin{eqnarray}\label{D-operator-dual-1}
\langle Df,\phi\rangle_{\lambda}=-\langle f,D\phi\rangle_{\lambda}.
\end{eqnarray}
A Laplace-type representation of the Dunkl kernel $E_\lambda(iz)$ is given by\,(\cite{ZhongKai Li 3},\cite{Ro1})
\begin{eqnarray}\label{D-kernel-3}
E_\lambda(iz)=c_{\lambda}'\int_{-1}^1e^{izt}(1+t)(1-t^2)^{\lambda-1}dt,
\ \ \ \ \ \
c'_{\lambda}=\frac{\Gamma(\lambda+1/2)}{\Gamma(\lambda)\Gamma(1/2)},
\end{eqnarray}
and $x\mapsto E_{\lambda}(ix\xi)$ is an eigenfunction of $D_x$ with eigenvalue $i\xi$\,(\cite{Du6}), that is
\begin{eqnarray}\label{D-kernel-2}
D_x[E_{\lambda}(ix\xi)]=i\xi E_{\lambda}(ix\xi), \ \ \ \ \ \
\ \ \ E_{\lambda}(ix\xi)|_{x=0}=1.
\end{eqnarray}
From Formulas\,(\ref{D-kernel-3},\ref{D-kernel-2}), an immediately case is that $e^{-y\xi}E_{\lambda}(ix\xi)$ is a $\lambda$-analytic function for a fixed given $\xi$,\  i.e., $T_{\bar{z}}e^{-y\xi}E_{\lambda}(ix\xi)=0$.

The Dunkl transform shares many  important  properties with the usual Fourier transform, part of which are listed as follows. These conclusions extend those on the Hankel transform and are special cases on the general Dunkl transform studied in\,(\cite{dJ},\cite{Du6}).

\begin{proposition}\label{D-transform-a}

{\rm(i)} \ If $f\in L_{\lambda}^1(\RR)$, then $\SF_{\lambda}f\in
C_0(\RR)$ and $\|\SF_{\lambda}f\|_{\infty}\leq\|f\|_{L_{\lambda}^1}$.

{\rm (ii)} \ {\rm (Inversion)} \ If $f\in
L_{\lambda}^1(\RR)$ such that $\SF_{\lambda}f\in L_{\lambda}^1(\RR)$, then
$f(x)=[\SF_{\lambda}(\SF_{\lambda}f)](-x)$.

{\rm(iii)} \ For $f\in\SS(\RR)$ or $f\in C_c^1(\RR)$, we have  $[\SF_{\lambda}(Df)](\xi)=i\xi(\SF_{\lambda}f)(\xi)$, $[\SF_{\lambda}(xf)](\xi)=i[D_{\xi}(\SF_{\lambda}f)](\xi)$ for $\xi\in\RR$, and $\SF_{\lambda}$ is a topological automorphism on $\SS(\RR)$.

{\rm(iv)} \ {\rm (Product Formula)} \ For all $f,g\in L_{\lambda}^1(\RR)$, we have
$\langle \SF_{\lambda}f,g\rangle_{\lambda}=\langle
f,\SF_{\lambda}g\rangle_{\lambda}$.

{\rm (v)} \ {\rm (Plancherel)} There exists a unique extension of
$\SF_{\lambda}$ to $L_{\lambda}^2(\RR)$ with
$\|\SF_{\lambda}f\|_{L_{\lambda}^2}=\|f\|_{L_{\lambda}^2}$, and
$(\SF_{\lambda}^{-1}f)(x)=(\SF_{\lambda}f)(-x)$.
\end{proposition}

\begin{corollary}\label{Hausdorff-Young} {\rm (Hausdorff-Young)}
For $1\le p\le2$, $\SF_{\lambda}f$ exists when $f\in L_{\lambda}^p(\RR)$, and

$\|\SF_{\lambda}f\|_{L_{\lambda}^{p'}}\le\|f\|_{L_{\lambda}^p}$,
where $1/p+1/p'=1$.
\end{corollary}

\subsection{The $\lambda$-Translation and the $\lambda$-Convolution}
For $x,t,z\in\RR$, we set
$W_{\lambda}(x,t,z)=W_{\lambda}^0(x,t,z)(1-\sigma_{x,t,z}+\sigma_{z,x,t}+\sigma_{z,t,x})$,
where
$$
W_{\lambda}^0(x,t,z)=\frac{c''_{\lambda}|xtz|^{1-2\lambda}\chi_{(||x|-|t||,
|x|+|t|)}(|z|)} {[((|x|+|t|)^2-z^2)(z^2-(|x|-|t|)^2)]^{1-\lambda}},
$$
$c''_{\lambda}=2^{3/2-\lambda}\big(\Gamma(\lambda+1/2)\big)^2/[\sqrt{\pi}\,\Gamma(\lambda)]$,
and  $\sigma_{x,t,z}=\frac{x^2+t^2-z^2}{2xt}$ for $x,t\in
\RR\setminus\{0\}$,
and 0 otherwise. From\,\cite{Ro1}, we have

\begin{proposition}\label{Dunkl-kernel-P}
The Dunkl kernel $E_\lambda$ satisfies the following product formula:
\begin{eqnarray}\label{product-D-1}
E_\lambda(x\xi)E_\lambda(t\xi)=\int_{\RR}E_\lambda(z\xi)d\nu_{x,t}(z),\quad
x,t\in\RR \ ~\hbox{and} ~ \ \xi\in\CC,
\end{eqnarray}
where $\nu_{x,t}$ is a signed measure given by
$d\nu_{x,t}(z)=c_{\lambda}W_\lambda(x,t,z)|z|^{2\lambda}dz$
for $x,t\in \RR\setminus\{0\}$, $d\nu_{x,t}(z)=d\delta_x(z)$ for $t=0$,
 and $d\nu_{x,t}(z)=d\delta_t(z)$ for $x=0$.
\end{proposition}
If $t\neq0$, for an appropriate function $f$ on $\RR$, the  $\lambda$-translation is given by
\begin{eqnarray}\label{tau-1}
(\tau_tf)(x)=c_{\lambda}\int_{\RR}f(z)W_\lambda(x,t,z)|z|^{2\lambda}dz;
\end{eqnarray}
and if $t=0$, $(\tau_0f)(x)=f(x)$.
If $(\tau_tf)(x)$ is taken as a function of $t$ for a given $x$,
we may set $(\tau_tf)(0)=f(t)$ as  a complement.
An unusual fact is that
$\tau_t$ is not a positive operator in general \cite{Ro1}.
If $(x, t)\neq(0,0)$, an equivalent form of $(\tau_tf)(x)$ is given by\,\cite{Ro1},
\begin{eqnarray}\label{tau-2}
(\tau_tf)(x)=c'_{\lambda}\int_0^\pi \bigg(f_e(\langle
x,t\rangle_\theta)+ f_o(\langle
x,t\rangle_\theta)\frac{x+t}{\langle
x,t\rangle_\theta}\bigg)(1+\cos\theta)\sin^{2\lambda-1}\theta
d\theta
\end{eqnarray}
where $x,t\in\RR$, $f_e(x)=(f(x)+f(-x))/2$,
$f_o(x)=(f(x)-f(-x))/2$, $\langle
x,t\rangle_\theta=\sqrt{x^2+t^2+2xt\cos\theta}$.
For two appropriated function $f$ and $g$, their $\lambda$-convolution
 $f\ast_{\lambda}g$ is defined by
\begin{eqnarray}\label{convolution-10}
(f\ast_{\lambda}
g)(x)=c_{\lambda}\int_{\RR}(\tau_xf)(-t)g(t)|t|^{2\lambda}dt.
\end{eqnarray}
The properties of $\tau$ and $\ast_{\lambda}$ are listed as follows\,\cite{ZhongKai Li 3}:

\begin{proposition}\label{tau-convolution-a} {\rm(i)} \ If $f\in L_{\lambda,{\rm loc}}(\RR)$,
then for all $x,t\in\RR$, $(\tau_tf)(x)=(\tau_xf)(t)$, and
$(\tau_t\tilde{f})(x)=(\widetilde{\tau_{-t}f})(x)$, where $\widetilde{f}(x)=f(-x)$.

{\rm(ii)}  \ For all $1\le p\le\infty$ and $f\in L_{\lambda}^p(\RR)$,
$\|\tau_tf\|_{L^p_{\lambda}}\le 4\|f\|_{L^p_{\lambda}}$ with $t\in\RR$, and for $1\leq p<\infty$, $\lim_{t\rightarrow0}\|\tau_tf-f\|_{L^p_{\lambda}}=0$.

{\rm(iii)} \ If $f\in L^p_\lambda(\RR)$, $1\leq p\leq2$ and $t\in\RR$, then
$[\SF_{\lambda}(\tau_t
f)](\xi)=E_{\lambda}(it\xi)(\SF_{\lambda}f)(\xi)$ for almost every $\xi\in\RR$.

{\rm(iv)} \ For measurable $f,g$ on $\RR$, if $\iint
|f(z)||g(x)||W_{\lambda}(x,t,z)||z|^{2\lambda}|x|^{2\lambda}dzdx$ is convergent, we have
  $\langle\tau_tf,g\rangle_{\lambda}=\langle
f,\tau_{-t}g\rangle_{\lambda}$. In particular,
$\ast_{\lambda}$ is commutative.

{\rm(v)} \ {\rm(Young inequality)} \ If $p,q,r\in[1,\infty]$ and
$1/p+1/q=1+1/r$, then for  $f\in L^p_{\lambda}(\RR)$, $g\in
L^q_{\lambda}(\RR)$, we have $\|f\ast_{\lambda}g\|_{L^r_{\lambda}}\leq
4\|f\|_{L^p_{\lambda}} \|g\|_{L^q_{\lambda}}$.

{\rm(vi)} \ Assume that $p,q,r\in[1,2]$, and $1/p+1/q=1+1/r$. Then for $f\in L^p_{\lambda}(\RR)$ and $g\in
L^q_{\lambda}(\RR)$,
$[\SF_{\lambda}(f\ast_{\lambda}g)](\xi)=(\SF_{\lambda}f)(\xi)(\SF_{\lambda}g)(\xi)$. In particular,
$\ast_{\lambda}$ is associative in $L^1_{\lambda}(\RR)$

{\rm(vii)} \ If $f\in L^1_{\lambda}(\RR)$ and $g\in\SS(\RR)$, then
$f\ast_{\lambda}g\in C^{\infty}(\RR)$.

{\rm(viii)} \ If $f,g\in L_{\lambda,{\rm loc}}(\RR)$, ${\rm
supp}\,f\subseteq\{x:\, r_1\le|x|\le r_2\}$ and ${\rm
supp}\,g\subseteq\{x:\,|x|\le r_3\}$,$r_2>r_1>0$, $r_3>0$, then
${\rm supp}\, (f\ast_{\lambda}g)\subseteq\{x:\,r_1-r_3\le|x|\le
r_2+r_3\}$.
\end{proposition}

\begin{proposition}(\hbox{\cite{ZhongKai Li 3}, Corollary\,2.5(i),\,Lemma\,2.7,\,Proposition\,2.8(ii)})\label{D-translation-a}

{\rm (i)} \ If $f\in\SS(\RR)$ or $\SD(\RR)$, then for fixed
$t$, the function
$x\mapsto(\tau_tf)(x)$ is also in $\SS(\RR)$ or $\SD(\RR)$, and
\begin{eqnarray}\label{D-translation-10}
 D_t(\tau_tf(x))=D_x(\tau_tf(x))=[\tau_t(Df)](x).
\end{eqnarray}
If $f\in\SS(\RR)$ and $m>0$, a pointwise estimate is given by
\begin{eqnarray}\label{D-translation-2}
\big|(\tau_t|f|)(x)\big|\le
\frac{c_m(1+x^2+t^2)^{-\lambda}}{(1+||x|-|t||^2)^m}.\end{eqnarray}

{\rm (ii)} \ If $\phi\in L_\lambda^1(\RR)$ satisfies $(\SF_{\lambda}\phi)(0)=1$ and $\phi_{\epsilon}(x)=\epsilon^{-2\lambda-1}\phi(\epsilon^{-1}x)$ for
$\epsilon>0$, then for all $f\in X= L_{\lambda}^p(\RR)$, $ 1\leq p<\infty$, or $C_0(\RR)$, $\lim_{\epsilon\rightarrow0+}\|f\ast_{\lambda}\phi_{\epsilon}-f\|_X=0$.

\end{proposition}

\subsection{The $\lambda$-Poisson Integral and its Conjugate Integral}
The following  results about  $\lambda$-Poisson Integral and Conjugate $\lambda$-Poisson  Integral are obtained in \cite{ZhongKai Li 3}.
For $1\leq p\leq\infty$, the $\lambda$-Poisson integral of $f\in L_\lambda^p(\RR)$ in the Dunkl setting is given by:
$$
(Pf)(x,y)=c_{\lambda}\int_{\RR}f(t)(\tau_xP_y)(-t)|t|^{2\lambda}dt,
\ \ \ \ \ \ \hbox{for} \ x\in\RR, \ y\in(0,\infty),
$$
where $(\tau_xP_y)(-t)$ is the $\lambda$-Poisson kernel with $P_y(x)=m_{\lambda}y(y^2+x^2)^{-\lambda-1}$,
$m_\lambda=2^{\lambda+1/2}\Gamma(\lambda+1)/\sqrt{\pi}$. Similarly,  the $\lambda$-Poisson integral for  $d\mu\in {\frak
B}_{\lambda}(\RR)$ can be given by $
(P(d\mu))(x,y)=c_{\lambda}\int_{\RR}(\tau_xP_y)(-t)|t|^{2\lambda}d\mu(t). $

\begin{proposition}\cite{ZhongKai Li 3}\label{Poisson-a} {\rm(i)} \
$\lambda$-Poisson kernel $(\tau_xP_y)(-t)$ can be represented by
\begin{eqnarray}\label{D-Poisson-ker-11}
(\tau_xP_y)(-t)=
\frac{\lambda\Gamma(\lambda+1/2)}{2^{-\lambda-1/2}\pi}\int_0^\pi\frac{y(1+{\rm
sgn}(xt)\cos\theta)
}{\big(y^2+x^2+t^2-2|xt|\cos\theta\big)^{\lambda+1}}\sin^{2\lambda-1}\theta
d\theta.
\end{eqnarray}

{\rm(ii)} \
The Dunkl transform of the function $P_y(x)$ is $(\SF_{\lambda}P_y)(\xi)=e^{-y|\xi|}$, and
$
(\tau_xP_y)(-t)=c_{\lambda}\int_{\RR}e^{-y|\xi|}E_{\lambda}(ix\xi)E_{\lambda}(-it\xi)|\xi|^{2\lambda}d\xi.
$
\end{proposition}
Relating to $P_y(x)$, the conjugate $\lambda$-Poisson integral for $f\in L^p_{\lambda}(\RR)$ is given by:
$$
(Qf)(x,y)=c_{\lambda}\int_{\RR}f(t)(\tau_xQ_y)(-t)|t|^{2\lambda}dt,
\ \ \ \ \ \ \hbox{for} \ x\in\RR, \ y\in(0,\infty),
$$
where $(\tau_xQ_y)(-t)$ is the conjugate $\lambda$-Poisson kernel with $Q_y(x)=m_{\lambda}x(y^2+x^2)^{-\lambda-1}$.

\begin{proposition}\cite{ZhongKai Li 3}\label{conjugate-Poisson-a} {\rm(i)} \ \
The conjugate  $\lambda$-Poisson kernel $(\tau_xQ_y)(-t)$ can be represented by
\begin{eqnarray*}\label{D-conjugate-Poisson-ker-1}
(\tau_xQ_y)(-t)=
\frac{\lambda\Gamma(\lambda+1/2)}{2^{-\lambda-1/2}\pi}\int_0^\pi\frac{(x-t)(1+{\rm
sgn}(xt)\cos\theta)
}{\big(y^2+x^2+t^2-2|xt|\cos\theta\big)^{\lambda+1}}\sin^{2\lambda-1}\theta
d\theta.
\end{eqnarray*}

{\rm(ii)} \  The Dunkl transform of $Q_y(x)$ is given by  $(\SF_{\lambda}Q_y)(\xi)=-i({\rm
sgn}\,\xi)e^{-y|\xi|}$, for $\xi\neq0$, and
$
(\tau_xQ_y)(-t)=-ic_{\lambda}\int_{\RR}({\rm
sgn}\,\xi)e^{-y|\xi|}E_{\lambda}(ix\xi)E_{\lambda}(-it\xi)|\xi|^{2\lambda}d\xi.
$
\end{proposition}

Then we can define the associated maximal functions as
$$(P^*_{\nabla}f)(x)=\sup_{|s-x|<y}|(Pf)(s,y)|,\ \ \
(P^*f)(x)=\sup_{y>0}|(Pf)(x,y)|,$$
$$(Q^*_{\nabla}f)(x)=\sup_{|s-x|<y}|(Qf)(s,y)|,\ \ \ (Q^*f)(x)=\sup_{y>0}|(Qf)(x,y)|.$$

\begin{proposition}\cite{ZhongKai Li 3}\label{Poisson-conjugate-CR}
{\rm (i)}    For $f\in L^p_{\lambda}(\RR)$, $1\le p<\infty$,  $u(x,y)=(Pf)(x,y)$ and  $v(x,y)=(Qf)(x,y)$
on $\RR^2_+$ satisfy the $\lambda$-Cauchy-Riemann equations\,(\ref{a c r0}), and are both $\lambda$-harmonic on $\RR^2_+$.\\
{\rm (ii)}\rm(semi-group property)  If $f\in L^p_{\lambda}(\RR)$, $1\leq p\leq \infty$, and $y_0>0$, then
$(Pf)(x,y_0+y)=P[(Pf)(\cdot,y_0)](x,y), \ \hbox{for} \ y>0.$\\
{\rm (iii)}  If $1 < p < \infty$, then there exists some constant $A_p'$ for any $f \in
L^p_{\lambda}(\RR)$, $\|(Q^*_{\nabla})f\|_{L^p_{\lambda}} \le A_p'
\|f\|_{L^p_{\lambda}}$.\\
{\rm (iv)}   $P^*_{\nabla}$ and $P^*$ are both $(p,p)$ type for $1<p
\leq\infty$ and weak-$(1,1)$ type.\\
{\rm (v)}    If $d\mu\in {\frak B}_{\lambda}(\RR)$, then
$\|(P(d\mu))(\cdot,y)\|_{L_{\lambda}^1}\le\|d\mu\|_{{\frak
B}_{\lambda}}$ as $y\rightarrow0+$, $[P(d\mu)](\cdot,y)$
converges $*$-weakly to $d\mu$: If $f\in X=L_{\lambda}^p(\RR)$, $1\le
p<\infty$, or $C_0(\RR)$, then $\|(Pf)(\cdot,y)\|_X\le\|f\|_X$ and
$\lim_{y\rightarrow0+}\|(Pf)(\cdot,y)-f\|_{X}=0$.\\
{\rm (vi)}  If $f\in L^p_{\lambda}(\RR)$, $1\le p\le2$, and
$[\SF_{\lambda}(Pf(\cdot,y))](\xi)=e^{-y|\xi|}(\SF_{\lambda}f)(\xi)$,
and
\begin{eqnarray}\label{D-Poisson-2}
(Pf)(x,y)=c_{\lambda}\int_{\RR}e^{-y|\xi|}(\SF_{\lambda}f)(\xi)E_{\lambda}(ix\xi)|\xi|^{2\lambda}d\xi,\quad
(x,y)\in\RR^2_+,
\end{eqnarray}
furthermore, Formula\,(\ref{D-Poisson-2}) is true when we replace $f\in L^p_{\lambda}(\RR)$ with $d\mu\in {\frak
B}_{\lambda}(\RR)$.
\end{proposition}

\subsection{The Theory of the Hardy space $H_{\lambda}^p(\RR_+^2)$ }
 The main results about the Hardy spaces $H_{\lambda}^p(\RR_+^2)$ obtained in\,\cite{ZhongKai Li 3} are listed as the following  Theorems\,\ref{majorization-2}--\ref{analytic-character-1} in this section.

\begin{theorem}\cite{ZhongKai Li 3} \label{majorization-2} {\rm ($\lambda$-harmonic majorization)}
For $F\in H^p_\lambda(\RR^2_+)$ and $p_0\le
p<\infty$ with $p_0=\frac{2\lambda}{2\lambda+1}$, there exists a nonnegative function $g\in
L^{p/p_0}_{\lambda}(\RR)$ such that for $(x, y)\in\RR^2_+$
\begin{align*}
& |F(x,y)|^{p_0}\leq (Pg)(x,y), \\
& \|F\|^{p_0}_{H^{p}_{\lambda}}\le\|g\|_{L^{p/p_0}_{\lambda}} \le
c_p\|F\|^{p_0}_{H^{p}_{\lambda}},
\end{align*}
where $c_p=2^{1-p_0+p_0/p}$ for $p_0\le p<1$ and $c_p=2$ for $p\ge1$.
\end{theorem}

\begin{theorem}\cite{ZhongKai Li 3} \label{ss}
For $p\ge p_0 =\frac{2\lambda}{2\lambda+1}$ and $F\in H_{\lambda}^p(\RR^2_+)$. Then\\
{\rm (i)} \  For a.e. $x\in\RR$, $\lim F(t,y)=F(x)$ exists as $(t,y)$
tends to $(x,0)$ nontangentially.\\
{\rm (ii)} \
If $F=0$ in a subset  $E\subset\RR$ with $|E|_{\lambda}>0$ symmetric about $x=0$, then $F\equiv0$.\\
{\rm (iii)} \  If $p>p_0$, then
$\lim_{y\rightarrow0+}\|F(\cdot,y)-F\|_{L^p_{\lambda}}=0$,
and $\|F\|_{H^{p}_{\lambda}}\approx\left(c_{\lambda}\int_{\RR}|F(x)|^{p}|x|^{2\lambda}dx\right)^{1/p}$.\\
{\rm (iv)} \  For $p >p_0$, $F\in H_\lambda^{p}(\RR^2_+)$ if and only if
$F^*_{\nabla}\in L^{p}_{\lambda}(\RR)$ and
   $ \|F\|_{H^{p}_{\lambda}}\approx \|F^*_{\nabla}\|_{L^p_{\lambda}}$, where   $F_{\nabla}^*(x)=\sup_{|x-u|<y}|F(u, y)|$ is the non-tangential maximal functions of $F$.\\
{\rm (v)} \   If $p >p_0$ and $p_1\geq p_0$, $F(x,y)\in H^{p}_{\lambda}(\RR^2_+)$
and $F(x) \in L^{p_1}_{\lambda}(\RR)$, then $F\in
H^{p_1}_{\lambda}(\RR^2_+)$.
\end{theorem}

\begin{theorem}\cite{ZhongKai Li 3}\label{analytic-character-1}
{\rm (i)} \    If $1\le p<\infty$ and  $F=u+iv\in H_{\lambda}^p(\RR^2_+)$, then $F$
is the $\lambda$-Poisson integral of its boundary values $F(x)$, and $F(x)\in
L^p_{\lambda}(\RR)$.\\
{\rm (ii)} \  If $1\le p\le2$, then the $\lambda$-Poisson integral of a function in $L^{p}_{\lambda}(\RR)$ is in $H^p_{\lambda}(\RR^2_+)$ if
and  only if the Dunkl transform of the function vanishes on $(-\infty, 0)$.\\
{\rm (iii)} \  If $p_0\le p\le1$ and $F(x,y)\in H_{\lambda}^p(\RR^2_+)$,
then $F$ has the following representation
\begin{eqnarray}\label{D-Poisson-5}
F(x,y)=c_{\lambda}\int_0^{\infty}e^{-y\xi}\phi(\xi)E_{\lambda}(ix\xi)|\xi|^{2\lambda}d\xi,\ \ \ \ (x, y)\in\RR_+^2,
\end{eqnarray}
where $\phi$ is a continuous function on $\RR$ satisfying $\phi(\xi)=0$ for $\xi\in(-\infty, 0]$ and that for each $y>0$, the function
$\xi\rightarrow e^{-y|\xi|}\phi(\xi)$ is bounded on $\RR$ and in  $L^{1}_{\lambda}(\RR)$.\\
{\rm (iv)} \  For $F\in H^p_{\lambda}(\RR^2_+)$, $p\ge p_0$,
there is a constant $c>0$, such that  $|F(x,y)|\leq
cy^{-(2\lambda+1)/p}\|F\|_{H^p_\lambda}$ for  $y>0$.\\
{\rm (v)} \  If $p_0<p\le2$,  $F\in H^p_{\lambda}(\RR^2_+)$, then $\int_0^{\infty}|(\SF_{\lambda}F)(\xi)|^p|\xi|^{(2\lambda+1)(p-2)+2\lambda}d\xi\le
c\|F\|_{H_{\lambda}^p}^p,$  where $c$ is a constant independent on $F$.\\
{\rm (vi)} \ For $F\in H^p_{\lambda}(\RR^2_+)$, where $p_0<p<1$ and $k\ge p$,  we could have $$\int_0^{\infty}|(\SF_{\lambda}F)(\xi)|^k|\xi|^{(2\lambda+1)(k-1-k/p)+2\lambda}d\xi\le
c\|F\|_{H_{\lambda}^p}^p,$$  $|(\SF_{\lambda}F)(\xi)|\le
c\xi^{(1+2\lambda)(p^{-1}-1)}\|F\|_{H_{\lambda}^p}$, and $(\SF_{\lambda}F)(\xi)=o\left(\xi^{(1+2\lambda)(p^{-1}-1)}\right)$
as $\xi\rightarrow+\infty$.
\end{theorem}

\subsection{The Real Hardy spaces $H_{\lambda}^{p}(\RR)$ for $1 \geq p>\frac{1}{1+\gamma_\lambda}$ with $\displaystyle{\gamma_\lambda=1/(4\lambda+2)}$ }
 The following  Definition\,\ref{o1}, Proposition\,\ref{s5}, Proposition\,\ref{s4}, Proposition\,\ref{s3} and Theorem\,\ref{u2} in this section can be found in \,\cite{hu} and\,\cite{MS2}.

A Lebesgue measure  function $a(x)$ is  a $p_{\lambda}$-atom, if it satisfies
 $\|a(x)\|_{L_{\lambda}^{\infty}}\lesssim \frac{1}{|I|_{\lambda}^{1/p}}$,
 ${\rm supp}\,a(x)\subseteq  I$, and
 $\int_{\RR} a(t)|t|^{2\lambda}dt=0 $, where $I$ is denoted as an interval on the real line $\RR$ and $|I|_{\lambda}=\int_I |t|^{2\lambda}dt$.

\begin{definition}\cite{hu}\label{o1}
For $\frac{2\lambda}{2\lambda+1}<p<\infty$, $\widetilde{H}_{\lambda}^p(\RR)$ is denote as
\begin{eqnarray*}
\widetilde{H}_{\lambda}^p(\RR)&\triangleq&\left\{g(x)\in L_{\lambda}^1(\RR)\bigcap L_{\lambda}^2(\RR): \|P_{\nabla}^*g\|^p_{L_{\lambda}^p(\RR)}<\infty\right\}.
\end{eqnarray*}
with:
$$\|g\|^p_{H_{\lambda}^p(\RR)}=\|P_{\nabla}^*g\|^p_{L_{\lambda}^p(\RR)}.$$
Thus for $p\geq1$, $\|\cdot\|_{H^p_{\lambda}(\RR)}$ is a norm and for $1>p>0$, $\|\cdot\|^p_{H_{\lambda}^p(\RR)}$ is a norm.

Thus  $\widetilde{H}_{\lambda}^p(\RR)$ is a linear space equipped with: $\|\cdot\|^p_{H_{\lambda}^p(\RR)}$ when $\frac{2\lambda}{2\lambda+1}<p<\infty$, and $\widetilde{H}_{\lambda}^p(\RR)$ is not a complete space. The completion of $\widetilde{H}_{\lambda}^p(\RR)$ with  $\|\cdot\|^p_{H_{\lambda}^p(\RR)}$ is denoted as  $H_{\lambda}^p(\RR)$.
\end{definition}

\begin{proposition}\cite{hu}\label{s5}
 $H_{\lambda}^p(\RR)\bigcap H_{\lambda}^2(\RR)\bigcap H_{\lambda}^1(\RR)$ is dense in $H_{\lambda}^p(\RR)$ for $\frac{2\lambda}{2\lambda+1}<p<\infty$.   $H_{\lambda}^p(\RR)=L^{p}_{\lambda}(\RR)$, for $1<p<\infty$.    $H_{\lambda}^1(\RR)\subset L^{1}_{\lambda}(\RR).$
\end{proposition}

\begin{proposition}\cite{hu}\label{s4}
For $1 \geq p>\frac{1}{1+\gamma_\lambda}$, $\displaystyle{\gamma_\lambda=\frac{1}{2(2\lambda+1)}}$, $H_{\lambda}^p(\RR)$ are  Homogeneous Hardy Spaces $H^p$. Thus by
\cite{MS2}, for any $f\in H_{\lambda}^p(\RR)$ there exists a sequence of $p_\lambda$-atoms $\{a_n(x)\}$ and a numerical sequence $\{\lambda_n\}$ such that
$$f(x)=\sum_n\lambda_n a_n(x)$$
in $H_{\lambda}^p(\RR)$ spaces, moreover, there exists two positive and finite constants $c_p$ and $C_p$ independent of $f$ such that
$$c_p\|f\|_{H_{\lambda}^p(\RR)}^p\leq \sum_n|\lambda_n|^p\leq C_p\|f\|_{H_{\lambda}^p(\RR)}^p.$$
Also, we could obtain the following Formula:
$$\|f\|_{H_{\lambda}^p(\RR)}^p\sim_{p} \sum_n|\lambda_n|^p\sim_{\lambda,p}\|P^*f\|^p_{L_{\lambda}^p(\RR)}.$$

\end{proposition}

\begin{proposition}\cite{MS2}\label{s6}
For $1 \geq p>0$,  Homogeneous Hardy spaces $H^p$ is the  linear space of all bounded linear functions $f$ on $Lip(1/p-1)$, which can be represented as
$$f(x)=\sum_n\lambda_n a_n(x)$$
in the sense of distributions on $Lip(1/p-1)$, where $\{a_n(x)\}$ is a sequence of $p_\lambda$-atoms and $\{\lambda_n\}$ is a numerical sequence with
$\sum_n|\lambda_n|^p<\infty$, and the norm of $H^p$ is given by
\begin{eqnarray}\label{usus1}
\|f\|_{H^p}= \inf \bigg\{\bigg(\sum_n|\lambda_n|^p\bigg)^{1/p}\bigg\}.
\end{eqnarray}

\end{proposition}

\begin{proposition}\cite{hu}\label{s3}
$u(x, y)$ is a $\lambda$-harmonic function, for $1 \geq p>\frac{1}{1+\gamma_\lambda}$\\
case1, $u_{\nabla}^\ast(x) \in L_{\lambda}^p(\RR)\bigcap L_{\lambda}^2(\RR)\bigcap L_{\lambda}^1(\RR)$, then there exists $f\in\widetilde{H}_{\lambda}^p(\RR)$, such that
\begin{eqnarray}\label{tan13}
u(x, y)= f*_\lambda P_y(x).
\end{eqnarray}
case2, $u_{\nabla}^\ast(x) \in L_{\lambda}^p(\RR)$, then there exists $f\in H_{\lambda}^p(\RR)$, such that
\begin{eqnarray}\label{tan14}
\displaystyle{\int \sup_{|x-s|<y}\bigg|u(s,y)-f*_\lambda P_y(s)\bigg|^p|x|^{2\lambda}dx=0,}
\end{eqnarray}
moreover, $$\|u_{\nabla}^\ast\|_{L_{\lambda}^p(\RR)}\sim \|f\|_{H_{\lambda}^p(\RR)}.$$
\end{proposition}

\begin{theorem}\cite{hu}\label{u2}
Let $u(x, y)$ to be a $\lambda$-harmonic function satisfying  $u_{\nabla}^* \in L^p_{\lambda}(\RR)$. For $\frac{2\lambda}{2\lambda+1}<p <\infty$, there exists a $\lambda$-analytic function $F(z)\in H_{\lambda}^p(\RR_+^2)$ satisfying $u(x, y)=ReF(z)$ and
$$\|F\|_{H_{\lambda}^p(\RR_+^2)}\sim \|u_{\nabla}^*\|_{L_{\lambda}^p(\RR)}.$$
\end{theorem}

\section{The $\lambda$-Hilbert transform}\label{Hilbert}
In this section, we will discuss the Definition of $\lambda$-Hilbert transform on the Real Hardy Spaces associated with the Dunkl setting $H_{\lambda}^{p}(\RR)$.

\begin{proposition}\cite{ZhongKai Li 3}\label{u1}\label{uu6}
For $1<p<\infty$, there is a constant $A_p$ such that for arbitrary $f(x)\in L_\lambda^p(\RR)$, $\|(Qf)(\cdot,y)\|_{L_\lambda^p(\RR)}\le A_p\|f\|_{L_\lambda^p(\RR)}$, there exists a function in $L_\lambda^p(\RR)$, denoted by $\SH_\lambda f$ such that
$(Qf)(\cdot,y)$ converge to $\SH_\lambda f$ as $y\rightarrow0+$, both in $L_\lambda^p(\RR)$ norm and almost everywhere nontangentially, and we have $\|\SH_\lambda f\|_{L_\lambda^p(\RR)}\le A_p\|f\|_{L_\lambda^p(\RR)}$. Moreover,
$$(Qf)(x,y)=[P(\SH_\lambda f)](x, y),$$
and for $1<p\leq2$,
$$[\SF_{\lambda}(\SH_\lambda f)](\xi)=-i(sgn \xi)(\SF_{\lambda}f)(\xi).$$

\end{proposition}

\begin{proposition}\cite{ZhongKai Li 3}\label{uu}
For $1\leq p<\infty$, $f(x)\in L_\lambda^p(\RR)$, its $\lambda$-Hilbert transform
$$(\SH_\lambda f)(x)=\lim_{y\rightarrow0+}(Qf)(x, y)$$
exists almost everywhere, and the mapping $f\rightarrow \SH_\lambda f$ is strongly-$(p, p)$ bounded for $1< p<\infty$ and weakly-$(1, 1)$ bounded.

\end{proposition}

\begin{proposition}\cite{ZhongKai Li 3} \label{Poisson-b}
For $f\in L_{\lambda}^1(\RR)\bigcap L_{\lambda}^{\infty}(\RR)$, $x\in\RR, \ y\in(0,\infty)$,  $\lambda$-Hilbert transform can be given by
\begin{eqnarray*}
\SH_\lambda f(x)=c_{\lambda}\lim_{\epsilon\rightarrow0+}\int_{|t-x|>\epsilon} f(t)h(x,t)|t|^{2\lambda}dt\ \ a.e.\ x\in\RR
\end{eqnarray*}
where
the $\lambda$-Hilbert kernel  is defined as :
$$h(x,t)=\frac{\lambda\Gamma(\lambda+1/2)}{2^{-\lambda-1/2}\pi}(x-t)\int_{-1}^1\frac{(1+s)(1-s^2)^{\lambda-1}
}{(x^2+t^2-2xts)^{\lambda+1}}ds.$$

\end{proposition}

\begin{definition}\label{u3}
For $\frac{2\lambda}{2\lambda+1}< p\leq1$, by Proposition\,\ref{s3} and Theorem\,\ref{u2}, we could know that for  any $f\in H_{\lambda}^{p}(\RR)\bigcap L_{\lambda}^{1}(\RR)\bigcap L_{\lambda}^{2}(\RR)$, there is  a $\lambda$-analytic function $F=u+iv$, such that $f$ and $u$ satisfy  Proposition\,\ref{s3}. By Theorem\,\ref{ss}, we could know that  $\lim F(t,y)=F(x)$ exists as $(t,y)$
tends to $(x,0)$ almost everywhere nontangentially.  By Proposition\,\ref{uu} and Theorem\, \ref{ss}\,{\rm (iii)} , we will define the $\lambda$-Hilbert transform of $f\in H_{\lambda}^{p}(\RR)\bigcap L_{\lambda}^{1}(\RR)\bigcap L_{\lambda}^{2}(\RR)$ as a function $\SH_\lambda f$ satisfying the following Formulas\,(\ref{uu1},\,\ref{uu7}):
\begin{eqnarray}\label{uu1}
(\SH_\lambda f)(x)=\lim_{(t, y) \underrightarrow{\angle}(x, 0+)}v(t, y),\ \ \ a.e.x\in\RR,
\end{eqnarray}
\begin{eqnarray}\label{uu7}
\lim_{y\rightarrow0+}\bigg\|u(\cdot,y)+iv(\cdot,y)-f(\cdot)-i(\SH_\lambda f)(\cdot)\bigg\|_{L^p_{\lambda}}=0.
\end{eqnarray}
\end{definition}

\section{The Boundedness of $\lambda$-Hilbert transform}\label{bounded Hilbert}
Notice that for $p>1$, $H_{\lambda}^{p}(\RR)=L_{\lambda}^{p}(\RR)$, thus the $\lambda$-Hilbert transform on $H_{\lambda}^{p}(\RR)$ is in fact on the spaces  $L_{\lambda}^{p}(\RR)$.  By Proposition\,\ref{u1}, we could know that  the $\lambda$-Hilbert transform is bounded on  $H_{\lambda}^{p}(\RR)$ spaces when $p>1$.

 For $1 \geq p>\frac{1}{1+\gamma_\lambda}$, the $H_{\lambda}^{p}(\RR)$ are Homogeneous Hardy Spaces, thus the $\lambda$-Hilbert transform  on $H_{\lambda}^{p}(\RR)$ may have connection to the $p_\lambda$-atoms. Thus we will give an estimation of $\lambda$-Hilbert transform of the $p_\lambda$-atoms.

\begin{definition}\label{u5}
For $\frac{1}{1+\gamma_\lambda}<p\leq 1$,  $\displaystyle{\gamma_\lambda=\frac{1}{2(2\lambda+1)}}$, $H_{\lambda, F}^p(\RR)$ is defined as following
\begin{eqnarray*}
H_{\lambda, F}^p(\RR)&\triangleq&\left\{f(x)=\sum_{n=1}^N\lambda_{n} a_n(x)\ in\ H_{\lambda}^p(\RR)\ spaces: \|f\|_{H_{\lambda}^p(\RR)}^p\sim\sum_{n=1}^N|\lambda_n|^p<\infty \ \hbox{for}\ any\ N\in\NN,\right.\\
 & &  \{a_n(x)\} \ are\    p_\lambda-atoms. \bigg\}
\end{eqnarray*}
\end{definition}

\begin{proposition}\label{u6}
For $\frac{1}{1+\gamma_\lambda}<p\leq 1$,  $\displaystyle{\gamma_\lambda=\frac{1}{2(2\lambda+1)}}$, $H_{\lambda, F}^p(\RR)$ is dense in
$H_{\lambda}^p(\RR)$, and we could have $H_{\lambda, F}^p(\RR)\subset H_{\lambda}^p(\RR)\bigcap L_{\lambda}^1(\RR) \bigcap L_{\lambda}^{2}(\RR) \subset H_{\lambda}^p(\RR)$.
\end{proposition}
\begin{proof}
It is easy to check that $H_{\lambda, F}^p(\RR) \subseteq L_{\lambda}^1(\RR)\bigcap L_{\lambda}^{2}(\RR)$, thus we could deduce that $H_{\lambda, F}^p(\RR)\subset H_{\lambda}^p(\RR)\bigcap L_{\lambda}^1(\RR) \bigcap L_{\lambda}^{2}(\RR) \subset H_{\lambda}^p(\RR)$.

By Proposition\,\ref{s4}, any $f\in H_{\lambda}^p(\RR)$ can be represented as
$$f(x)=\sum_{n=1}^{\infty}\lambda_{n} a_n(x)\ \ \hbox{in\ $H_{\lambda}^p(\RR)$\ spaces},$$
where $\{a_n(x)\}$ \ are\    $p_\lambda$-atoms, and $\{\lambda_n\}$ is a numerical sequence with
$\sum_n|\lambda_n|^p<\infty$. For any $m\in\NN$, we use $f_m(x)$ to denote as:
$f_m(x)=\sum_{n=1}^{m}\lambda_{n} a_n(x)$ in\ $H_{\lambda}^p(\RR)$\ spaces.
Thus it is easy to check that $f_m(x)\in  H_{\lambda}^p(\RR)$, then $f-f_m\in H_{\lambda}^p(\RR)$. Thus by Proposition\,\ref{s4} and Formula\,(\ref{usus1}), we could deduce that
\begin{eqnarray}\label{us4}
\|f-f_m\|_{H_{\lambda}^p(\RR)}\lesssim \sum_{i=m+1}^{\infty} |\lambda_i|^p.
\end{eqnarray}
With the fact $ \sum_{n=1}^{\infty}|\lambda_n|^p<\infty$, thus for any $\varepsilon>0$, $\exists N_\varepsilon>0$, such that for any $m>N_\varepsilon$
\begin{eqnarray}\label{us5}
 \sum_{i=m+1}^{\infty} |\lambda_i|^p\lesssim \varepsilon.
\end{eqnarray}
By Formula\,({\ref{us4}}) and Formula\,({\ref{us5}}), we could deduce that for any $m>N_\varepsilon$,
$$\|f-f_m\|_{H_{\lambda}^p(\RR)}\lesssim \varepsilon.$$
Then there exists some sufficiently large number $M$ depending only on $c_p$ and $C_p$ in Proposition\,\ref{s4},  such that when $\displaystyle{ \varepsilon< \frac{\|f\|_{H_{\lambda}^p(\RR)}}{M}}$,  for any $m>N_\varepsilon$, the following Formula holds:
$$\sum_{i=1}^{m} |\lambda_i|^p\sim \sum_{i=1}^{\infty} |\lambda_i|^p\sim \|f\|_{H_{\lambda}^p(\RR)}\sim\|f_m\|_{H_{\lambda}^p(\RR)}.$$
That is
$$\sum_{i=1}^{m} |\lambda_i|^p\sim\|f_m\|_{H_{\lambda}^p(\RR)}.$$
Thus we could deduce that $f_m\in H_{\lambda, F}^p(\RR)$ if  $m\in \NN$ and $m$ is sufficiently large. Thus $H_{\lambda, F}^p(\RR)$ is dense in
$H_{\lambda}^p(\RR)$.
\end{proof}

\begin{proposition}\label{uu2}
For $1<p<\infty$, $f\in L_{\lambda}^p(\RR)$, we have the following equations:
\begin{eqnarray}\label{uu4}
P(\SH_\lambda f)(x, y)=Qf(x, y),
\end{eqnarray}
\begin{eqnarray}\label{uu5}
Q(\SH_\lambda f)(x, y)=-Pf(x, y).
\end{eqnarray}
And $P(\SH_\lambda f)(x, y)+iQ(\SH_\lambda f)(x, y)$ is a $\lambda$-analytic function.
\end{proposition}
\begin{proof}
Formula\,(\ref{uu4}) is deduced from Proposition\,\ref{u1} directly. Then we will prove  Formula\,(\ref{uu5}). By Proposition\,\ref{uu6}, we could deduce  that   $\lambda$-Hilbert transform is a bounded map from $L_{\lambda}^p(\RR)$ to $L_{\lambda}^p(\RR)$ for $1<p<\infty$. Thus $(\SH_\lambda f)(x)\in L_{\lambda}^p(\RR)$ if $f\in L_{\lambda}^p(\RR)$. By Proposition\,\ref{Poisson-conjugate-CR}{\rm (i)}, $P(\SH_\lambda f)(x, y)+iQ(\SH_\lambda f)(x, y)$ is a $\lambda$-analytic function. We could also notice that $Qf(x, y)-iPf(x, y)$ is a $\lambda$-analytic function. We use $\widetilde{F}(z)$ to denote as $\widetilde{F}(z)= P(\SH_\lambda f)(x, y)+iQ(\SH_\lambda f)(x, y)$  and $F(z)$ to denote as $F(z)=Qf(x, y)-iPf(x, y)$.
Thus by Formula\,(\ref{uu4}), we could obtain
$$Re\left(\widetilde{F}(z)-F(z)\right)=0.$$
Thus  $\widetilde{F}(z)-F(z)$ is a $\lambda$-analytic function with the real parts $0$. Then we could deduce that $$\widetilde{F}(z)-F(z)=iC,$$
for some constant $C$. Notice that the operators $P$, $Q$ and $\SH_\lambda$ are  bounded maps from $L_{\lambda}^p(\RR)$ to $L_{\lambda}^p(\RR)$ for $1<p<\infty$, thus we could deduce that the function $x\rightarrow\widetilde{F}(z)-F(z)\in L_{\lambda}^p(\RR)$. Then $iC\in L_{\lambda}^p(\RR) $ only when
$C=0$. This proves the Proposition.
\end{proof}

\begin{proposition}\label{uu2*}
 $F(x, y)=u(x, y)+iv(x, y))\in H_{\lambda}^p(\RR_+^2)$ with $\frac{2\lambda}{2\lambda+1}<p<\infty$, we have the following:
$$\|v_{\nabla}^*\|_{L_{\lambda}^p(\RR)}\sim\|u_{\nabla}^*\|_{L_{\lambda}^p(\RR)}\sim\|F\|_{H_{\lambda}^p(\RR_+^2)}.$$
\end{proposition}
\begin{proof}
Notice that $\widetilde{F}(z)= v(x, y)-iu(x, y)\in H_{\lambda}^p(\RR_+^2)$ and
$$\|F\|_{H_{\lambda}^p(\RR_+^2)}=\|\widetilde{F}\|_{H_{\lambda}^p(\RR_+^2)}.$$
Thus by Theorem\,\ref{u2} we could deduce that
$$\|v_{\nabla}^*\|_{L_{\lambda}^p(\RR)}\sim\|\widetilde{F}\|_{H_{\lambda}^p(\RR_+^2)}=\|F\|_{H_{\lambda}^p(\RR_+^2)}\sim\|u_{\nabla}^*\|_{L_{\lambda}^p(\RR)}.$$
This proves the Proposition.
\end{proof}

\begin{theorem}\label{t1}
For $\frac{2\lambda}{2\lambda+1}<p<\infty$, $\lambda$-Hilbert transform is a bounded map from $H_{\lambda}^p(\RR)$ to $L_{\lambda}^p(\RR)$ for any function $f \in H_{\lambda}^p(\RR)\bigcap L_{\lambda}^{1}(\RR)\bigcap L_{\lambda}^{2}(\RR)$:
$$\|\SH f\|_{L_{\lambda}^p(\RR)}\lesssim\|f\|_{H_{\lambda}^p(\RR)}.$$
 Thus $\lambda$-Hilbert transform can be extended as a bounded operator from $H_{\lambda}^p(\RR)$ to $L_{\lambda}^p(\RR)$.
\end{theorem}
\begin{proof}
For $\frac{2\lambda}{2\lambda+1}< p\leq1$, for  any $f\in H_{\lambda}^p(\RR)\bigcap L_{\lambda}^{1}(\RR)\bigcap L_{\lambda}^{2}(\RR)$, there is  a $\lambda$-analytic function $F=u+iv$, such that $f$ and $u$ satisfy  Proposition\,\ref{s3}, $u$ and $F$ satisfy Theorem\,\ref{u2}.
From Formula\,(\ref{uu1}) and Formula\,(\ref{uu7}), we could obtain the following inequality
\begin{eqnarray}\label{uu3}
\|\SH_\lambda f\|_{L_{\lambda}^p(\RR)}=\lim_{y\rightarrow0+}\|v(,y)\|_{L_{\lambda}^p(\RR)}\leq\|v_{\nabla}^*\|_{L_{\lambda}^p(\RR)}.
\end{eqnarray}
Thus by Proposition\,\ref{uu2*} and  Formula\,(\ref{uu3}), we could deduce that
$$ \|\SH_\lambda f\|_{L_{\lambda}^p(\RR)}\lesssim\|v_{\nabla}^*\|_{L_{\lambda}^p(\RR)}\lesssim\|F\|_{H_{\lambda}^p(\RR_+^2)}\sim\|f\|_{H_{\lambda}^p(\RR)}.$$
Thus we could have
$$ \|\SH_\lambda f\|_{L_{\lambda}^p(\RR)}\lesssim\|f\|_{H_{\lambda}^p(\RR)}.$$
 Notice that $ H_{\lambda}^p(\RR)\bigcap L_{\lambda}^{1}(\RR)\bigcap L_{\lambda}^{2}(\RR)$ is dense in  $ H_{\lambda}^p(\RR)$, thus $\lambda$-Hilbert transform can be extended as a bounded operator from $H_{\lambda}^p(\RR)$ to $L_{\lambda}^p(\RR)$.
\end{proof}

\begin{proposition}\label{estimate a}
 $$\int_{-1}^{1}  \left(1-bs\right)^{-\lambda-1}(1+s)(1-s^2)^{\lambda-1}ds \leq C \frac{1}{1-|b|} ,\ \  \forall -1< b<1 \ ,\ \ \lambda>0$$
 C is depend on $\lambda$, and independent on $b$. ($C\thicksim 1/\lambda$)
\end{proposition}

\begin{proof}
CASE 1: when $0\leq b<1$.\\
It is obvious to see that when $0\leq b<1 $,
$\int_{-1}^{0}  \left(1-bs\right)^{-\lambda-1}(1+s)(1-s^2)^{\lambda-1}ds \lesssim 1 .$ Then, by the Formula of integration by parts and  $1-s\leq1-bs$ when $1\geq s\geq0$, we obtain:
\begin{eqnarray*}
\left|\int_{0}^{1} \left(1-bs\right)^{-\lambda-1}(1+s)(1-s^2)^{\lambda-1}ds\right|
 &\lesssim&
 \left|\int_{0}^{1}   \left(1-bs\right)^{-\lambda-1}(1-s)^{\lambda-1}ds\right|
\\&\lesssim&\frac{1}{\lambda}+\frac{\lambda+1}{\lambda}b\int_{0}^{1}\left(1-bs\right)^{-2}ds
\\&\lesssim&
\frac{1}{1-b}.
\end{eqnarray*}

CASE 2: when $-1<b\leq0$.\\
Obviously we could have$\int_{0}^{1}  \left(1-bs\right)^{-\lambda-1}(1+s)(1-s^2)^{\lambda-1}ds \lesssim 1 .$  Then,
by the Formula of integration by parts and  $1+s\leq1-bs$ when $-1\leq s\leq0$, we obtain:
\begin{eqnarray*}
 \left|\int_{-1}^{0} \left(1-bs\right)^{-\lambda-1}(1+s)(1-s^2)^{\lambda-1}ds\right|
 &\lesssim&
 \left|\int_{-1}^{0}   \left(1-bs\right)^{-\lambda-1}(1+s)^{\lambda}ds\right|
\\&\lesssim&\frac{1}{\lambda+1}-b\int_{-1}^{0}\left(1-bs\right)^{-1}ds
\\&\lesssim&
-\ln(1+b)
\\&\lesssim&
\frac{1}{1+b}.
\end{eqnarray*}
By CASE 1 and CASE 2, we could obtain
$$\int_{-1}^{1}  \left(1-bs\right)^{-\lambda-1}(1+s)(1-s^2)^{\lambda-1}ds \leq C \frac{1}{1-|b|} ,\ \  \forall -1< b<1 \ ,\ \ \lambda>0.$$
 This proves the proposition.
\end{proof}

\begin{lemma}\cite{ZhongKai Li 3}
  Let $c>1$ be a fixed number. Then for $x,t,t', s\in\RR$ and $|s|\leq1$, if $||x|-|
t||>c\delta, |t-t'|<\delta$ with $\delta>0$, we have
\begin{equation}\label{Basic-Es0}|x|^2+|t|^2-2xts\thickapprox |x|^2+|t'|^2-2xt's\end{equation}
\end{lemma}

\begin{proposition}\label{us8}
$a(x)$ is a $p_\lambda$-atom, then we  could deduce that $a(x)\in H_{\lambda}^p(\RR)$, moreover the function $P(a+i\SH_\lambda a)(x, y)$ is a $\lambda$-analytic function and $P(a+i\SH_\lambda a)(x, y)\in H_{\lambda}^p(\RR_+^2)$ for the range of $\frac{1}{1+\gamma_\lambda}<p\leq1$, with $\displaystyle{\gamma_\lambda=1/(4\lambda+2)}$, further more, we could have
$$\|a\|^p_{H_{\lambda}^p(\RR)}\lesssim_{\lambda,p} C$$
for some constant $C$.
\end{proposition}
\begin{proof}
 Notice that $H_\lambda^p(\RR)=L_\lambda^p(\RR)$ when $1<p<\infty$, thus we need only to prove the case when $\frac{1}{1+\gamma_\lambda}<p\leq1$.
Let $a(x)$ to be   a $p_{\lambda}$-atom, with
 $\|a(x)\|_{L_{\lambda}^{\infty}}\lesssim \frac{1}{|I|_{\lambda}^{1/p}}$,
 ${\rm supp}\,a(x)\subseteq  I$, and
 $\int_{\RR} a(t)|t|^{2\lambda}dt=0 $, where $I$ is denoted as the interval $(t_0-\delta, t_0+\delta)$ on the real line $\RR$. We use
 $\widetilde{I}$ to denote as the interval $(-t_0-\delta, -t_0+\delta)$, $I_c$ to denoted as the  $(t_0-c\delta, t_0+c\delta)$  and
 $\widetilde{I}_c$ to denote as the interval $(-t_0-c\delta, -t_0+c\delta)$, and $I_0=(-c\delta, +c\delta)$, where $c>0$ is a fixed constant.
Notice that $|I_0|_\lambda\sim\int_0^{c\delta}|t|^{2\lambda}dt\lesssim\int_{|t_0|}^{|t_0|+c\delta}|t|^{2\lambda}dt\sim |I_c|_{\lambda}\sim \int_{|t_0|/c}^{|t_0|/c+\delta}|t|^{2\lambda}dt\lesssim \int_{|t_0|}^{|t_0|+\delta}|t|^{2\lambda}dt\sim |I|_{\lambda}$ holds for $c\geq1$, and $|I_0|_\lambda\sim\int_0^{c\delta}|t|^{2\lambda}dt\lesssim\int_{|t_0|}^{|t_0|+c\delta}|t|^{2\lambda}dt\sim |I_c|_{\lambda}\lesssim \int_{|t_0|}^{|t_0|+\delta}|t|^{2\lambda}dt\sim |I|_{\lambda}$ holds for $c\leq1$, thus  we could have
$$|I_0|_{\lambda}\lesssim_\lambda |I_c|_{\lambda}\lesssim_\lambda |I|_{\lambda}.$$

By Proposition\,\ref{s4},
$\|a\|^p_{H_{\lambda}^p(\RR)}=\|P^*a\|^p_{L_{\lambda}^p(\RR)}$. Then we have
\begin{eqnarray*}
\|a\|^p_{H_{\lambda}^p(\RR)}&=&\|P^*a\|^p_{L_{\lambda}^p(\RR)}=\int_{\RR}  \sup_{y>0}|\left(a*_{\lambda}P_y\right)(x)|^p |x|^{2\lambda}dx
\\&=&\int_{I_c\bigcup \widetilde{I}_c\bigcup I_0}   \sup_{y>0}|\left(a*_{\lambda}P_y\right)(x)|^p |x|^{2\lambda}dx    \\
&+& \int_{\left(I_c\bigcup \widetilde{I}_c\bigcup I_0\right)^c}   \sup_{y>0}|\left(a*_{\lambda}P_y\right)(x)|^p |x|^{2\lambda}dx \\
&=&I+II.
\end{eqnarray*}
$\mathbf{CASE}$ 1: when $x\in I_c\bigcup \widetilde{I}_c\bigcup I_0$.\\
By  Holder inequality, Proposition\,\ref{Poisson-conjugate-CR}(iv) and $\|a(x)\|_{L_{\lambda}^{\infty}}\lesssim \frac{1}{|I|_{\lambda}^{1/p}}$,  we could  obtain:
\begin{eqnarray}\label{8}
I&=&\int_{I_c\bigcup \widetilde{I}_c\bigcup I_0}   \sup_{y>0}|\left(a*_{\lambda}P_y\right)(x)|^p |x|^{2\lambda}dx  \nonumber \\
&\leq& \left(\int_{I_c\bigcup \widetilde{I}_c \bigcup I_0}   \sup_{y>0}|\left(a*_{\lambda}P_y\right)(x)|^2 |x|^{2\lambda}dx\right)^{p/2}
\left(\int_{I_c\bigcup \widetilde{I}_c\bigcup I_0}  |x|^{2\lambda}dx\right)^{1-(p/2)} \nonumber \\
&\leq& C.
\end{eqnarray}
$\mathbf{CASE}$ 2: when $x\in \left(I_c\bigcup \widetilde{I}_c\bigcup I_0\right)^c$ .\\
When $x\in \left(I_c\bigcup \widetilde{I}_c\bigcup I_0\right)^c$, we could write $\left(a*_{\lambda}P_y\right)(x)$ as:
$$\left(a*_{\lambda}P_y\right)(x)=c_{\lambda}\int a(t)(\tau_xP_y)(-t)|t|^{2\lambda}dt .$$
The Formula of $\lambda$-Poisson Kernel can be written as
\begin{eqnarray}\label{us3}
(\tau_xP_y)(-t)&=&
\frac{\lambda\Gamma(\lambda+1/2)}{2^{-\lambda-1/2}\pi}\int_0^\pi\frac{y(1+{\rm
sgn}(xt)\cos\theta)
}{\big(y^2+x^2+t^2-2|xt|\cos\theta\big)^{\lambda+1}}\sin^{2\lambda-1}\theta
d\theta
\nonumber\\&=& \frac{\lambda\Gamma(\lambda+1/2)}{2^{-\lambda-1/2}\pi}\int_{-1}^1y\frac{(1+s)(1-s^2)^{\lambda-1}
}{(y^2+x^2+t^2-2xts)^{\lambda+1}}ds.
\end{eqnarray}
For convenience, we use $\langle\cdot ,\cdot\rangle_{y,s}$ to denote as:
$\langle x,t\rangle_{y,s}=y^2+x^2+t^2-2xts$, with $y>0$, $-1\leq s\leq1$.
By the vanish property of the $a(t)$, we could write $\left(a*_{\lambda}P_y\right)(x)$ as,
\begin{eqnarray}\label{us1}
\left(a*_{\lambda}P_y\right)(x)=c_{\lambda}\int a(t)\left((\tau_xP_y)(-t)-(\tau_xP_y)(-t_0)\right)|t|^{2\lambda}dt.
\end{eqnarray}
By the mean value theorem and Formula\,(\ref{Basic-Es0}), we could deduce that:
\begin{eqnarray}\label{us2}
\left|\frac{y}{\langle x,t\rangle_{y,s}^{\lambda+1}}-\frac{y}{\langle x,t_0\rangle_{y,s}^{\lambda+1}}\right|&=&\nonumber\left|y\frac{\langle x,t_0\rangle_{y,s}^{\lambda+1}-\langle x,t\rangle_{y,s}^{\lambda+1}}{\langle x,t\rangle_{y,s}^{\lambda+1}\langle x,t_0\rangle_{y,s}^{\lambda+1}}\right|
\\ &\lesssim& \nonumber \left|y|t-t_0|\frac{\langle x,t_0\rangle_{y,s}^{\lambda+(1/2)}}{\langle x,t_0\rangle_{y,s}^{2\lambda+2}}\right|
\\ &\lesssim&  \left|\frac{|t-t_0|}{\langle x,t_0\rangle_{y,s}^{\lambda+1}}\right|.
\end{eqnarray}
Thus Formulas\,(\ref{us1}, \ref{us2}, \ref{us3}) lead to
\begin{eqnarray}
\left|\left(a*_{\lambda}P_y\right)(x)\right|\leq C |I|_{\lambda}^{1-(1/p)} \int_{-1}^{1} \frac{|t-t_0|}{\left(\langle x,t_0\rangle_s\right)^{\lambda+1}} (1+s)(1-s^2)^{\lambda-1}ds.
\end{eqnarray}
By Proposition\,\ref{estimate a}, we obtain the following
\begin{eqnarray}\label{5}
\int_{-1}^1\frac{(1+s)(1-s^2)^{\lambda-1}
}{(x^2+t_0^2-2xt_0s)^{\lambda+1}}ds
&=&
\left( \frac{1}{x^2+t_0^2}\right)^{\lambda+1}\int_{-1}^1(1+s)(1-s^2)^{\lambda-1}
\left(1-\frac{2xt_0 s}{x^2+t_0^2 }\right)^{-\lambda-1}ds  \nonumber\\
&\leq&
C\frac{1}{(x^2+t_0^2)^{\lambda} (|x|-|t_0|)^2}.
\end{eqnarray}
Thus we could have
\begin{eqnarray}\label{106}
\left|\left(a*_{\lambda}P_y\right)(x)\right|\leq C |I|_{\lambda}^{1-(1/p)} \frac{\delta}{\left||x|-|t_0|\right|^{2}\left||x|+|t_0|\right|^{2\lambda}}
\end{eqnarray}
It is easy to see that
\begin{eqnarray}\label{106*}
|I|_\lambda\lesssim\left(|t_0|+\delta\right)^{2\lambda}\delta\leq\left(|t_0|+|x|\right)^{2\lambda}\delta,\ \ \hbox{for}\,x\in \left(I_c\bigcup \widetilde{I}_c\bigcup I_0\right)^c.
\end{eqnarray}
 By Formula\,(\ref{106}) and Formula\,(\ref{106*}), we could deduce that
\begin{eqnarray}\label{111}
\int_{\left(I_c\bigcup \widetilde{I}_c\bigcup I_0\right)^c} |\left(a*_{\lambda}P_y\right)(x)|^p |x|^{2\lambda}dx&\lesssim_{\lambda}& \left(\delta^{2p-1}\right)\int_{\delta}^{+\infty}\frac{1}{r^{2p}}dr
\nonumber\\&\lesssim_{\lambda,p}&C.
\end{eqnarray}
Thus by Formula\,(\ref{8}) and Formula\,(\ref{111}), we could deduce that:
$$\|a\|^p_{H_{\lambda}^p(\RR)}\lesssim_{\lambda,p} C.$$
By the fact $a(t)\in L_\lambda^2(\RR)$ and Proposition\,\ref{Poisson-conjugate-CR}, we could deduce that
the function $P(a+i\SH_\lambda a)(x, y)$ is a $\lambda$-analytic function. Thus by Proposition\,\ref{s3} and Theorem\,\ref{u2}, we could deduce that
$P(a+i\SH_\lambda a)(x, y)\in H_{\lambda}^p(\RR_+^2)$ for the range of $\frac{1}{1+\gamma_\lambda}<p\leq1$, with $\displaystyle{\gamma_\lambda=\frac{1}{2(2\lambda+1)}}$. This proves the Proposition.
\end{proof}

\begin{proposition}\label{us6}
$a(x)$ is a $p_\lambda$-atom. For $\frac{1}{1+\gamma_\lambda}<p\leq1$, with $\displaystyle{\gamma_\lambda=\frac{1}{2(2\lambda+1)}}$, we could deduce that $\SH_\lambda a(x)\in H_\lambda^p(\RR)$, furthermore we could deduce the following
$$\|\SH_\lambda a\|_{H_\lambda^p(\RR)}\sim\|a\|_{H_\lambda^p(\RR)}.$$
\end{proposition}
\begin{proof}Obviously $a(x)\in L_\lambda^2(\RR)$ when $a(x)$ is a $p_\lambda$-atom. Thus by Proposition\,\ref{uu2}, we could deduce that
$$P(\SH_\lambda a)(x, y)+iQ(\SH_\lambda a)(x, y) =Qa(x, y)-iPa(x, y).$$
Thus we could have
$$\|P(\SH_\lambda a)(\cdot, y)+iQ(\SH_\lambda a)(\cdot, y)\|_{H_\lambda^p(\RR_+^2)}=\|Pa(\cdot, y)+iQa(\cdot, y)\|_{H_\lambda^p(\RR_+^2)}.$$
Together with Proposition\,\ref{s3} and Theorem\,\ref{u2}, we could deduce that
$$\|\SH_\lambda a\|_{H_\lambda^p(\RR)}\sim\|P(\SH_\lambda a)(\cdot, y)+iQ(\SH_\lambda a)(\cdot, y)\|_{H_\lambda^p(\RR_+^2)}=\|Pa(\cdot, y)+iQa(\cdot, y)\|_{H_\lambda^p(\RR_+^2)}\sim\|a\|_{H_\lambda^p(\RR)}.$$
This proves the Proposition.
\end{proof}

\begin{proposition}\label{us7}
For $\frac{1}{1+\gamma_\lambda}<p\leq 1$,  $\displaystyle{\gamma_\lambda=\frac{1}{2(2\lambda+1)}}$, for any $f\in H_{\lambda, F}^p(\RR)$, we could deduce that
$$\|\SH_\lambda f\|_{H_\lambda^p(\RR)}\lesssim\|f\|_{H_\lambda^p(\RR)}.$$
\end{proposition}
\begin{proof}
For any $f\in H_{\lambda, F}^p(\RR)$, we could write $f$ as
\begin{eqnarray*}
f(x)=\sum_{n=1}^N\lambda_{n} a_n(x)\ in\ H_{\lambda}^p(\RR)\ spaces, where\  \|f\|_{H_{\lambda}^p(\RR)}^p\sim\sum_{n=1}^N|\lambda_n|^p<\infty,\,and\,\{a_n(x)\}\, are\, p_\lambda-atoms.
\end{eqnarray*}
Thus by Proposition\,\ref{us8} and Proposition\,\ref{us6}, we could deduce that
\begin{eqnarray*}
\|\SH_\lambda f\|_{H_{\lambda}^p(\RR)}^p\lesssim \sum_{n=1}^{N} |\lambda_n|^p\|\SH_\lambda a_n\|_{H_{\lambda}^p(\RR)}^p\lesssim \sum_{n=1}^{N} |\lambda_n|^p\sim\| f\|_{H_{\lambda}^p(\RR)}^p.
\end{eqnarray*}
This proves the Proposition.
\end{proof}
By Proposition\,\ref{u6} and Proposition\,\ref{us7}, we could obtain our main theorem:
\begin{theorem}\label{us9}
For $\frac{1}{1+\gamma_\lambda}<p\leq 1$,  $\displaystyle{\gamma_\lambda=\frac{1}{2(2\lambda+1)}}$, $\lambda$-Hilbert transform can be extended as a bounded operator from $H_{\lambda}^p(\RR)$ to $H_{\lambda}^p(\RR)$:
$$\|\SH_\lambda f\|_{H_{\lambda}^p(\RR)}\lesssim\|f\|_{H_{\lambda}^p(\RR)},\ \ \hbox{for\,any\,}f\in H_{\lambda}^p(\RR).$$

\end{theorem}


\begin{thebibliography}{}






\bibitem{dJ}  M.\,F.\,E. de Jeu, The Dunkl transform, \it Invent. Math. \bf 113\rm(1993), 147-162.

\bibitem{Du3} C.\,F. Dunkl, Differential-difference  operators  associated  to reflection groups,
                \it Trans. Amer. Math. Soc. \bf 311\rm(1989), 167-183.



\bibitem{Du6} C.\,F. Dunkl, Hankel transforms associated to finite reflection groups, in ``Proc. of the special session on hypergeometric
               functions on domains of positivity, Jack polynomials and applications (Tampa,
                1991)", \it Contemp. Math. \bf 138\rm(1992), 123-138.







\bibitem{BGS} D.L.Burkholder, R.F.Gundy, and M.L.Silverstein, The maximal function characterization of the class $H^p(\RR)$, \it Trans. Amer. Math. Soc, \bf157\rm(1971), 137-153.




\bibitem{ZhongKai Li 1}Zh-K Li, Hardy Spaces for Jacobi expansions, Analysis \bf16\rm(1996), 27-49.


\bibitem{ZhongKai Li 2}Zh-K Li and J-Q Liao, Hardy Spaces For Dunkl-Gegenbauer Expansions, J.Funct.Anal.\bf265\rm(2013), 687-742.

\bibitem{ZhongKai Li 3}Zh-K Li and J-Q Liao, Harmonic Analysis Associated with One-dimensional Dunkl Transform, \it J.Approx.\bf37\rm(2013), 233-281.


\bibitem{Ro1} M. R\"osler, Bessel-type signed hypergroups on $\RR$, in ``Probability measures on groups and related structures XI", Edited
               by H. Heyer and A. Mukherjea, World Scientific, Singapore, 1995, \it pp. \rm 292-304.



\bibitem{MS}  B. Muckenhoupt and E.\,M. Stein, Classical expansions and their relation to conjugate harmonic functions, \it Trans. Amer. Math.
                Soc. \bf 118\rm(1965), 17-92.




\bibitem{MS2} R. Macias and C. Segovia, A decomposition into atoms of distributions on spaces of homogeneous type, \it Adv. in Math.
               \bf 33\rm(1979), 271-309.


\bibitem{hu}ZhuoRan Hu, Hardy spaces  associated with One-dimensional Dunkl transform for  $\frac{2\lambda}{2\lambda+1}<p\leq1$, submitted\, (preprint).



\end{thebibliography}
\end{document}